\documentclass[a4paper,10pt]{amsart}
\usepackage[utf8]{inputenc}
\usepackage{amsfonts, amsmath, amssymb}
\usepackage[left=3cm,right=3cm]{geometry}
\usepackage{amsthm}
\usepackage{graphicx}

\newtheorem{tma}{Theorem}[section]

\newtheorem{defn}{Definition}
\newtheorem{lema}{Lemma}

\newtheorem{step}{Step}

\title{Smoothening cone points with Ricci flow}

\author[Daniel Ramos]{Daniel Ramos \\ \\ \tiny{September 26, 2011}}

\begin{document}

\begin{abstract}
We consider Ricci flow on a closed surface with cone points. The main result is: given a (nonsmooth) cone metric $g_0$ over a closed surface there is a smooth Ricci flow $g(t)$ defined for $(0,T]$, with curvature unbounded above, such that $g(t)$ tends to
$g_0$ as
$t\rightarrow 0$. This result means that Ricci flow provides a way for instantaneously smoothening cone points.
We follow the argument of P. Topping in \cite{Top1} modifying his reasoning for cusps of negative curvature; in that sense we can
consider cusps as a limiting zero-angle cone, and we generalize to any angle between 0 and $2\pi$.
\end{abstract}

\thanks{Universitat Autònoma de Barcelona, Departament de Matemàtiques - 80193 Bellaterra, Barcelona (Spain)- \\ E-mail: \texttt{dramos@mat.uab.cat} . }

\maketitle

\section{Introduction}
Ricci flow on closed surfaces was studied first by R. Hamilton \cite{Ham1} and B. Chow \cite{Chow1}, proving that any smooth closed riemannian surface $(\mathcal M, g_0)$ admits a volume-normalized Ricci flow $g(t)$, $t\in[0,T]$, with uniformly bounded curvature, having $g_0$ as initial condition, $g(0)=g_0$. This flow is unique and well defined and converges to a constant curvature metric, \cite{Chenlutian} (See also \cite{ChowKnopf}). The unnormalized flow may develop finite-time singularities in the case of a sphere, when the curvature tends globaly to infinity as well as the diameter tends to zero. Some analogous results were obtained for Ricci flow on orbifold surfaces by L-F. Wu and Chow \cite{Wu}, \cite{Chowwu}, \cite{Chow2}. They assume an equivariant definition of the Ricci flow under the action of the isotropy group of the cone points. Therefore, the only nontrivial case are bad orbifolds, which do not admit a
smooth manifold as global branched covering space, so the Ricci flow cannot be lifted there. They prove that bad orbifolds (the teardrop
and the football) admit a (normalized) Ricci flow converging to a soliton solution.

The alternative consideration of the Ricci flow just acting on the smooth part of the orbifold leads to consider the Ricci flow on an
open, noncomplete manifold, which does not fit in the classical theory of Hamilton, so existence and uniqueness might be lost. H. Yin
obtained, however, an existence theorem for Ricci flow on cone surfaces \cite{Yin}, with uniformly bounded curvature, defined on the
smooth part of the surface, and preserving the conical structure of each singular point. This is the analogous to the classical and
orbifold Ricci flow. On a different approach, Topping and G. Giesen \cite{Topgie} obtained an existence theorem for Ricci flow on
incomplete surfaces, which becomes instantaneously complete and has uniformly bounded curvature, which exposes the nonuniquenes of
solutions. 

In another work, Topping \cite{Top1} considered a complete open surface with cusps of negative curvature and proved the
existence of a instantaneously smooth Ricci flow with unbounded curvature, a ``smoothening flow'' which erases instantaneously the cusps.
This requires a generalized notion of initial metric for a flow, that we will use thorough the paper:

\begin{defn}[Cf. \cite{Top1} Definition 1.1]\label{def1}
Let $\mathcal M$ be a smooth manifold, and $p_1,\ldots,p_n\in\mathcal M$. Let $g_0$ be a riemannian metric on $\mathcal
M\setminus\{p_1,\ldots,p_n\}$ and let $g(t)$ be a smooth Ricci flow on $\mathcal M$ for $t\in(0,T]$. We say that $g(t)$ has initial
condition $g_0$ if 
$$g(t) \longrightarrow g_0\ \mathrm{as}\ t\rightarrow 0$$
smoothly locally on $\mathcal M\setminus\{p_1,\ldots,p_n\}$.
\end{defn}

The technique for this result consists in capping the cusps of the original metric $g_0$ with a smooth part near the cusp point, in an
increasing sequence of metrics, each term with a further and smaller capping. This sequence of smooth metrics gives rise to a sequence of
(classical) Ricci flows, and the work consists in proving that this sequence has a limiting Ricci flow on $\mathcal M$ which has $g_0$ as
initial condition in the sense of Definition \ref{def1}. Our work proves that this technique works equally well on cone surfaces, using
truncated or ``blunt'' cones as approximations for a cone point. In our setting, cusps would be seen as a limiting case of a zero-angle
cone. This provides an instantaneously smooth Ricci flow that smooths out the cone points of a cone surface.

The paper is organized as follows: in sections \ref{sectionconepoints} and \ref{sectionRFconesurf} we review the notions of cone surface and
Ricci flow on cone surfaces, and state the two main theorems of the paper (existence and uniqueness of the smoothening flow). In section
\ref{sectiontruncating} we build the truncated cones that will serve us as
approximations of a cone point and in section \ref{sectionbarriers} we build upper barriers that, applied to our truncated cones, will give
us control on the
convergence of the sequence. In section \ref{sectionexistproof} we put together the preceding lemmas to prove the existence theorem; and
finally in section \ref{sectionuniqueproof}
we prove the uniqueness theorem.

\emph{Acknowledgements:} The author was partially supported by Feder/Micinn through the Grant MTM2009-0759 and by the Fundaci\'o Ferran Sunyer i Balaguer. He also wishes to thank his advisor, Joan Porti, for all his guidance.

\section{Cone points}\label{sectionconepoints}
Cone surfaces are topological surfaces equipped with a riemannian metric which is smooth everywhere except on some discrete set of points
(\emph{cone points}) that look like the vertex of a cone. Typical examples include orbifolds, where a group acting by isometries leads to
identification of different directions as seen from a fixed point. In the case of two dimensions, orientable orbifolds consist locally in
the quotient of a smooth manifold by perhaps the action of a cyclic group acting by rotations, leading to the rise of singular points at the
center of the rotations. The space of directions is no longer a metric circle of length $2\pi$ but a metric circle of length
$\frac{2\pi}{n}$ (this is the cone angle). General 2-dimensional cone points include all angles, not only submultiples of $2\pi$, although
we will restrict our attention to cone angles less than $2\pi$. A useful fact in the case of two dimensions is that an isolated cone point
$p$
has a neighbourhood $U$ which admits a chart with coordinates on the unit disc $D$. Precisely, there is a diffeomorphism $U\setminus\{p\}
\rightarrow D\setminus\{0\}$ that allows us to define a metric tensor on the coordinates of $D$ (which will be undefined on the origin). We
may even take isothermal coordinates $(x,y)$ (or in polar coordinates $(r,\theta)$, or in complex notation $z=x+iy$), that is, a local chart
where  the metric can be written as
$$g=e^{2u}(dx^2 + dy^2)=e^{2u}(dr^2 + r^2 d\theta^2)=e^{2u}|dz|^2$$
with $u$ a real-valued function of the coordinates (possibly undefined in $z=0$). 
The main example is the metric of the euclidean cone. Consider $\mathbb R^2$ and trace two half-lines from the origin, meeting at angle
$\alpha$. Consider the metric space resulting of identification of the two half lines, and let us bring it a riemannian metric. Using polar
coordinates, assume that there exist a riemannian metric of the form
$$g=\phi^2(r)\ (dr^2+r^2 d\theta^2)$$
where $\phi$ should depend also on $\theta$, but we try just depending on $r$ due to symmetry. We look for a function $\phi$ that produces a
flat metric (zero curvature) and a cone angle at the origin. Let us consider a parallel curve $r=x=const$. On the one hand, its length is 
the
angle times the radius (since it is a region of the plane),
$$L=\alpha \int_0^x \phi(r) dr.$$
On the other hand, the length of the curve measured on the metric is
$$L=\int_0^{2\pi} x \phi(x) d\theta = 2\pi x \phi(x).$$
So
$$\alpha \int_0^x \phi(r) dr = 2\pi x \phi(x)$$
and denoting $\Phi'(r)=\phi(r)$ we obtain
$$\alpha \Phi(x) = 2\pi x \Phi'(x)$$
and solving this ODE,
$$\phi(r) = \Phi'(r) = \frac{\alpha}{2\pi} r^{\frac{\alpha}{2\pi}-1}$$
Renaming $\beta =\frac{\alpha}{2\pi}-1$, with $-1<\beta\leq 0$, the cone metric of the euclidean cone is
$$g=(\beta+1)^2 r^{2\beta} (dr^2 + r^2 d\theta^2)$$

Having seen the prototype of cone point, the following definition is justified:
\begin{defn}
A cone surface $(\mathcal M, (p_1,\ldots,p_n),g)$ is a topological surface $\mathcal M$ and $p_1,\ldots, p_n\in \mathcal M$ equipped
with a smooth riemannian metric $g$ on $\mathcal M\setminus \{p_1,\ldots, p_n\}$, such that every point $p_i$ admits an open
neighbourhood $U_i$, and diffeomorphism $U_i\setminus \{p_i\} \rightarrow D\setminus \{0\}$ where the metric on the
coordinates of $D\setminus\{0\}$ is written as
$$g=e^{2(a_i+\beta_i \ln r)} |dz|^2$$
where $a_i:D \rightarrow \mathbb R$ is a bounded and continuous function on the whole disc, and $-1<\beta_i\leq0$.
\end{defn}
The \emph{cone angle} at $p_i$ is $\alpha_i:=2\pi(\beta_i+1)$. We say that $\mathcal M$ has bounded curvature if it has bounded
riemannian curvature on the smooth part of $ \mathcal M$, although it has $+\infty$ curvature in the sense of Alexandrov at the cone
points (provided the angle is less than $2\pi$).

\section{Ricci flow on cone surfaces}\label{sectionRFconesurf}

A Ricci flow is an evolution equation for the metric tensor of a manifold, $\frac{\partial}{\partial t}g =-2 Ric$. In the 2-dimensional case, the equation is 
$$\frac{\partial}{\partial t}g = -2 K g$$
where $K$ is the Gauss curvature of the surface. We may take a chart with isothermal coordinates $g=e^{2u}|dz|^2$ and on that setting the Ricci flow equation turns
$$\frac{\partial}{\partial t}u = e^{-2u}\Delta u = -K$$
where $\Delta$ is the usual euclidean laplacian.

The existence of solutions for this equation (at least for a short time) is given by the classical theorems on Ricci flow when the initial metric is smooth over a closed surface. However, cone surfaces do not fit on these theorems since the flow is defined on an open noncomplete manifold and the cone point (the origin of the coordinate chart) acts as a sort of boundary of the domain. We will appeal to a theorem by H. Yin (\cite{Yin} Theorem 1.1) that ensures the existence for a short time of at least one solution for the Ricci flow on cone surfaces, and such flow preserves the cone angles of each cone point of the surface. J. Isenberg, R. Mazzeo, and N. Sesum \cite{IseMazSes} have annonced another approach to the existence of the flow with cone singularities. What we are constructing on this paper is a different solution to the same equation: a Ricci flow which is instantaneously smooth for any $t>0$, that satisfies the Ricci equation for any $t\in(0,T]$, and that converges to the initial nonsmooth cone metric as $t\rightarrow 0$. Despite the nonuniquenes shown by this result, there is certain uniqueness provided we restrict to certain class of flows.

Now we state the two main theorems of the paper:
\begin{tma}\label{tma1}
Let $(\mathcal M,(p_1\ldots p_n),g_0)$ be a closed cone surface; with bounded curvature. There exists a Ricci flow $g(t)$ smooth on the
whole $\mathcal M$, defined for $t\in (0,T]$ for some $T$, and such that
$$g(t)\underset{t\rightarrow 0}{\longrightarrow} g_0.$$
Furthermore, this Ricci flow has curvature unbounded above and uniformly bounded below over time.
\end{tma}

\begin{tma}\label{tma2}
Let $\tilde g(t)$ be a Ricci flow on $\mathcal M$, defined for $t\in (0,\delta]$ for some $\delta < T$, such that
$$\tilde g(t)\underset{t\rightarrow 0}{\longrightarrow} g_0$$
and assume that its Gauss curvature is uniformly bounded below. Then $\tilde g(t)$ agrees with the flow $g(t)$ constructed in Theorem
\ref{tma1} for $t\in (0,\delta]$.
\end{tma}

\section{Truncating cones}\label{sectiontruncating}
This section is analogous to section 3.3 of \cite{Top1}, where we substitute the cusp points with cone points. Let $D$ denote the unit disc, and $r=|z|$. An appropriate elimination
of the asymptote of the conformal factor at $r=0$ gives rise to a metric which smooth, and no longer singular at the origin.
\begin{lema}\label{lema1}
Let $g_0=e^{2(a_0+\beta \ln r)}|dz|^2$ be a cone metric on the punctured disc $D\setminus\{0\}$ with curvature bounded below, $K[g_0]\geq
-\Lambda$. There exists an increasing sequence of smooth metrics $g_k= e^{2u_k}|dz|^2$ on $D$  such that 
 \begin{enumerate}
  \item $g_k=g_0$ on $D\setminus D_{1/k}$,
  \item $g_k \leq g_0$ on $D\setminus \{0\}$,
  \item $\inf_{D_{1/k}}u_k \rightarrow + \infty$ as $k\longrightarrow +\infty$, and
  \item $K[g_k]\geq \min\{e^2 K[g_0],0\}$.
 \end{enumerate}
\end{lema}

\begin{proof}
The conformal factor $u_0=a_0+\beta \ln r$ of the cone metric tends to $+\infty$ as
$r\rightarrow 0$, so for each $k\in \mathbb N$ we pick
the minimum of $u_0$ and $k$ to obtain an increasing sequence of bounded functions tending to $u_0$. This has to be done in a way such
that the functions remain smooth.

Choose a smooth function $\psi : \mathbb R \rightarrow \mathbb R$ such that 
\begin{itemize}
 \item $\psi(s)=s$ for $s\leq -1$,
 \item $\psi(s)=0$ for $s\geq 1$,
 \item $\psi'\geq 0$ and $\psi''\leq 0$.
\end{itemize}
The smoothed minimum of $u_0$ and $k$ is
$$u_k=\psi(u_0-k)+k$$
and satisfies
\begin{itemize}
 \item If $u_0\geq k+1$ then $u_k=k$ and therefore $K[g_k]=0$.
 \item If $u_0\leq k-1$ then $u_k=u_0$ and therefore $K[g_k]=K[g_0]$.
 \item If $k-1<u_0<k+1$ then \begin{itemize}
                              \item[\textbullet] $u_k\leq u_0$,
			      \item[\textbullet] $u_k\leq k$,
			      \item[\textbullet] $u_k \geq u_0-1$.
                             \end{itemize}
\end{itemize}
So 
$$u_k\leq \min\{u_0,k\}$$
and then (2) and (3) are satisfied. We can compute
$$\Delta u_k = \psi''(u_0-k)|\nabla u_0|^2 + \psi'(u_0-k)\Delta u_0 \leq \psi'(u_0-k)\Delta u_0$$
Now, since $\psi'\geq 0$, we can distinguish 
$$\Delta u_k \leq \Delta(u_0)\ \mathrm{if}\ \Delta u_0 >0$$
or
$$\Delta u_k\leq 0 \ \mathrm{if}\ \Delta u_0 \leq 0$$
So
$$\Delta u_k \leq  \max \{\Delta u_0, 0\}$$
and then
$$K[g_k]=-e^{-2u_k}\Delta u_k \geq \min\{e^2 K[g_0],0\}$$
so (4) is satisfied. Finally (1) is satisfied after passing to a subsequence, since the region of points $\{z:u_0(z)>k\}$ shrinks to a point when $k\rightarrow \infty$. 
\end{proof}

\section{Upper barriers}\label{sectionbarriers}
The conformal factor of a cone surface possesses asymptotes at the coordinates of the cone points, whereas the truncated approximations have a finite but probably big value on that coordinates. This section provides a ratio of how fast the maximum value of this conformal factors decay as the Ricci flow evolves.

\begin{lema}\label{lema2}
 Let $g(t)=e^{2u(t)}|dz|^2$ be a smooth Ricci flow on $D$ and $t\in [0,\delta]$, and assume that
 $$u(0)\leq A + \beta \ln r$$
 for some $A\in \mathbb R$. Then 
 $$u(t)< B + \frac{\beta}{2(\beta+1)}\ln t$$
for some $B$ depending only on $A$ and $\beta$.
\end{lema}

\begin{proof}
We will consider the conformal factor of several different surfaces.
The function 
$$s(r):=\ln\left(\frac{2}{1+r^2}\right)$$ 
is the conformal factor of a sphere, and the functions
$$v_0(r):=\ln(2(\beta+1))+\beta\ln r$$
$$v_1(r):=\ln(2(\beta+1))+\beta\ln r - \ln\left(1-r^{2(\beta+1)}\right)$$
are the conformal factors of euclidean and hyperbolic cones (curvature $0$ and $-1$) respectively. Note that the euclidean and hyperbolic cones become indistinguishable as $r\rightarrow 0$.

Considering the Ricci flow ($\frac{\partial u}{\partial t} = e^{-2u}\Delta u=-K$) on the hyperbolic cone, it evolves as
$$V_1(t)=v_1 + c(t)$$
with $c(t)$ an increasing function, so comparing with say $t=1$, we have
$$V_1(t)<v_1 + C$$
for some constant $C$ and for all $0<t<1$.

The function $s\left(\frac{r}{c_1}\right) + c_2$
is the conformal factor of a rescaled sphere (in parameter and in metric). We define
$$U(r,t):=\left\{ 
\begin{array}{lc} 
S(r,\lambda(t)):=s\left(\frac{r}{\lambda}\right) + v_1(\lambda) + C & \mathrm{if}\ 0<r\leq \lambda \\
v_1(r) + C & \mathrm{if}\ \lambda<r<1 
\end{array}\right.
$$
where $\lambda= \lambda(t)$ is a function of $t$ to be determined. Geometrically, $U$ is the conformal factor of a piecewise smooth
metric, a hemisphere near the origin and a cone with constant negative curvature away from it. It is a kind of ``blunt cone'', the transition being at
coordinate
$r=\lambda(t)$. We
still have to determine $\lambda(t)$, but we will require it to tend to $0$ as $t\rightarrow 0$.
In order to prove the lemma we will see (a) $u \leq U$ and (b) $\sup U(\cdot,t) \leq B + \frac{\beta}{2(\beta+1)}\ln t$.

We prove (a). We can assume that $u(0)<v_1+C$, and we know that at $r=0$ the value of $u$ is finite. Since the capping of $S(r,\lambda)$ occurs at arbitrarily big values, it is also true that $u(0)<U(0)$. Indeed, for $0<r\leq \lambda$, we have $S(r,\lambda)\geq v_1(\lambda)+C \rightarrow +\infty$ as $t\rightarrow 0$ since
$\lambda\rightarrow 0$. So $u< U$ for small positive $t$.

Suppose that for some $t_0$ there is a $0<r_0<1$ such that $u(r_0,t_0)=U(r_0,t_0)$. We can assume $t<1$. Note that the asymptote of $U$ at $r=1$ avoids the case of $r_0=1$.
If the point occurs at $\lambda\leq r_0 <1$, then $u$ would be touching the upper barrier of $V_1(t)$, which is impossible since by the maximum principle $u$ cannot pass over $V_1$.

Assume then that $0<r_0<\lambda$. We have $U-u\geq 0$ for $0\geq t \geq t_0$ and
$$u(r_0,t_0)=U(r_0,t_0),\qquad \frac{\partial}{\partial t}(U-u)\bigg|_{r_0,t_0}\leq 0,\qquad \Delta(U-u)\bigg|_{r_0,t_0}\geq 0$$
so at $(r_0,t_0)$
$$0\geq \frac{\partial U}{\partial t} - \frac{\partial u}{\partial t} = \frac{\partial U}{\partial t} - e^{-2u}\Delta u = \frac{\partial
U}{\partial t} + e^{-2U}(\Delta(U-u)-\Delta U) \geq \frac{\partial U}{\partial t} - e^{-2U}\Delta U$$
so
$$\frac{\partial U}{\partial t} \leq e^{-2U}\Delta U.$$
We now choose $\lambda(t)$ properly to contradict this assertion. 
On the one hand, at $(r_0,t_0)$
$$\frac{\partial U}{\partial t} = \frac{\partial S}{\partial t} = \left( - s'\left(\frac{r}{\lambda}\right) \frac{r}{\lambda^2} +
\frac{\beta}{\lambda} + \frac{2(\beta+1)\lambda^{2(\beta+1)}}{(1-\lambda^{2(\beta+1)})\lambda}\right) \frac{\partial \lambda}{\partial t} \geq \frac{\beta}{\lambda}\frac{\partial \lambda}{\partial t}$$
since $s'(r)<0$. 
On the other hand, one can compute
$$e^{-2U}\Delta U = e^{-2S}\Delta S = -\frac{\lambda^{-2(\beta+1)}}{(\beta+1)^2} \frac{e^{-2C}}{4} (1-\lambda^{2(\beta+1)})^2.$$
Ignoring the negligible term tending to zero (geometrically, assuming a flat cone), one can guess a critical value of $\lambda$ by solving
$$\frac{\beta}{\lambda}\frac{\partial \lambda}{\partial t} = -\frac{\lambda^{-2(\beta+1)}}{(\beta+1)^2} \frac{e^{-2C}}{4}$$
e.g. with the solution
$$\lambda(t)=\left(\frac{-te^{-2C}}{2\beta(\beta+1)}\right)^{\frac{1}{2(\beta+1)}}.$$
A slight modification, say
$$\bar\lambda(t)=\left(\frac{-te^{-2C}}{4\beta(\beta+1)}\right)^{\frac{1}{2(\beta+1)}},$$
gives
$$\frac{\partial S}{\partial t}(r,\bar\lambda) \geq
\frac{\beta}{\bar\lambda}\frac{\partial \bar\lambda}{\partial t} =
\frac{1}{2}\frac{\beta}{(\beta+1)t} >
\frac{\beta}{(\beta+1)t} \left(1+\frac{te^{-2C}}{4\beta(\beta+1)}\right)=
e^{-2S(r,\bar\lambda)}\Delta S(r,\bar\lambda),$$
giving a contradiction as long as $C$ is big enough. Therefore there is no such time $t_0$ and so $u\leq U$.

Now we prove (b). We use the $\lambda=\bar \lambda$ just found. It is easy to check that $S(r,t)$ is nonincreasing and has a maximum at
$r=0$. Its value is
$$S(0,\bar\lambda(t))=\ln(4(\beta+1)) + \beta \ln\left(\left(\frac{-t}{4(\beta+1)\beta}\right)^{\frac{1}{2(\beta+1)}}\right) -\ln\left( 1+ \frac{te^{-2C}}{4\beta(\beta+1)}\right) \leq B + \frac{\beta}{2(\beta+1)}\ln t.$$
\end{proof}

\section{Proof of the existence Theorem}\label{sectionexistproof}
We now prove the theorem \ref{tma1}

\begin{proof}
For simplicity assume there is just one cone point $p$. We take isothermal coordinates $z$ on a neighbourhood of $p$ such that $p$
corresponds to $z=0$, and $z\in D$ the unit disc (rescaling parameter and metric if necessary), so the metric on this chart has the form
$$g_0=e^{a+\beta\ln r}|dz|^2$$
with $a:D\rightarrow \mathbb R$ a bounded continuous function.

We truncate the metric $g_0$ as in Lemma \ref{lema1} and we obtain an increasing sequence of smooth metrics $g_k$ on $\mathcal M$ such that:
\begin{enumerate}
  \item $g_k=g_0$ on $D\setminus D_{1/k}$,
  \item $g_k \leq g_0$ on $D\setminus \{0\}$,
  \item $\inf_{D_{1/k}}u_k \rightarrow + \infty$ as $k\longrightarrow +\infty$,
  \item $K[g_k]\geq \min\{e^2 K[g_0],0\}$ and
 \item $g_k\leq g_{k+1}$.
 \end{enumerate}
We apply Ricci flow to each initial metric $g_k$ and obtain a sequence of flows $g_k(t)$. There exist a uniform $T>0$ such that all flows
$g_k(t)$ are defined for $t\in[0,T]$. Indeed, by \cite{ChowKnopf} in dimension 2, if $\chi(\mathcal M)<2$ the flow is defined for
$t\in[0,\infty)$, and if $\chi(\mathcal M)=2$ the flow is defined for $t\in[0,\frac{\mathrm{Area}(\mathcal M)}{8\pi})$, and as $g_k\leq g_{k+1}$, then $\mathrm{Area}_{k}\leq\mathrm{Area}_{k+1}$.
So in any case the area does not tend to zero.

By the maximum principle, the initial $g_k(0)\leq g_{k+1}(0)$ implies $g_k(t)\leq g_{k+1}(t)$
and again by the maximum principle, $K[g_k(0)]\geq -\Lambda$ implies $K[g_k(t)]\geq -\Lambda $.

There exists also $g_s(t)$, the Yin's Ricci flow on $\mathcal M\setminus\{p\}$, and since $g_k(0)\leq g_0 = g_s(0)$, by the maximum principle we have
$$g_k(t) \leq g_{k+1}(t) \leq g_s(t),$$
so we can define the limit flow
$$G(t)=\lim_{k\rightarrow \infty} g_k(t).$$

On any chart not containing $p$, the flow $G(t)$ is smooth by the uniform bounds of $g_k$ and the parabolic regularity theory. We need to
ensure that $G(t)$ extends smoothly to $p$. It is enough to show that the conformal factor of $G(t)$ in a
neighbourhood of $p$ does not tend to $\infty$ for $t>0$. We use the Lemma \ref{lema2}. Say $G(t)=e^{2v(t)}|dz|^2$, then
$$v(t)=\lim_{k\rightarrow \infty} u_k(t) \leq C + \frac{\beta}{2(\beta+1)}\ln t$$
so $v(t)<+\infty$ for all $t>0$.
Furthermore, the uniform lower bound of the curvature on the approximant terms $g_k(t)$ also passes to the limit, so $K[G(t)] > -\Lambda$.
\end{proof}

\section{Uniqueness}\label{sectionuniqueproof}
The uniqueness issue is parallel to Topping's cusps, so we will sketch the proof and refer to \cite{Top1} for a detailed completion. Although there are two Ricci flows with a cone surface as initial metric, say Yin's flow and our constructed smoothening flow, Yin's flow is unique amongst the bounded curvature, cone-singular flows; and our flow is unique amongst the lower-bounded curvature, instantaneously-smooth flows.

\begin{proof}(Theorem \ref{tma2}) Recall that $\tilde g(t)$ is a Ricci flow defined on $\mathcal M$ for $t\in(0,\delta]$, with curvature uniformly bounded below, and such that $\tilde g(t)\rightarrow g_0$ as $t\rightarrow 0$. We want to show that it is unique. The proof consists in 4 steps:
\begin{step}
There exists a neighbourhood $\Omega$ of $p_i$, where the metric is written $\tilde g(t)=e^{2u}|dz|^2$, and there exists
$m\in \mathbb R$ such that $$u\geq m$$
in $\Omega$ for $t\in (0,\frac{\delta}{2}]$.
\end{step}
This step makes use of the lower curvature bound. Since $\frac{\partial u}{\partial t} = e^{-2u}\Delta u = -K[\tilde g] < \Lambda$, we have
$$u(z,t)\geq u(z,\frac{\delta}{2})-\Lambda\left(\frac{\delta}{2}-t\right)\geq \inf_{\Omega}u(\cdot,\frac{\delta}{2})-\Lambda\frac{\delta}{2}=:m.$$

\begin{step}
Actually, for every $M<\infty$, there is a small enough neighbourhood $\Omega_1$ and a small enough time $\delta_1$ such that
$$u\geq M$$
in $\Omega_1$ for $t\in (0,\delta_1)$.
\end{step}
This bound is obviously true for the conformal factor $u_0$ of the metric $g_0$, since $u_0=a+\beta\ln r$ has an asymptote on $r=0$. However, it is not clear that the factors $u(t)$ of the metrics $\tilde g(t)$ remain bounded by an arbitrary constant on a small neighbourhood for small $t$. It might happen that the functions $u(t)\rightarrow u_0$ as $t\rightarrow 0$ with $u(t)$ fixed at $r=0$ (that is, non-uniform convergence); but this case would contradict the uniform bounded below curvature. The sketch of the proof is as follows.

Define the family of functions $h(t)=\max\{M-u(t),0\}$, and the goal is proving that $h(t)\equiv 0$ for all $t<\delta_1$. We do that by showing that its $L^1$ norm on some small disc, $||h(t)||=\int_{D_\epsilon} |h(t)| d\mu$, vanishes. For, on the one hand $||h(t)||\rightarrow 0$ as $t\rightarrow 0$, since
$$||h(t)||=\max\{M-u(t),0\}\rightarrow \max\{M-u(0),0\} =0$$
because $u_0>M$. On the other hand, we claim that $\frac{d}{dt}||h(t)||\leq 0$, what proves the result. In order to prove that claim, we change the functions $h(t)$ by a smoothed version of the maximum, in a similar fashion we did in the proof of Lemma \ref{lema1}, that is $\hat h_\rho (t)=\Psi_\rho(M-u)\rightarrow h(t)$ as $\rho\rightarrow 0$. This allows us to compute $\frac{d}{dt}||\hat h_\rho(t)||$ in terms of the derivatives of the controlled function $\Psi_\rho$, the lower bound on $u(t)$ given by the previous step, and the lower curvature bound. See \cite{Top1} for the details.

\begin{step}
With the lower bound of $u$, we can compare the flow $\tilde g(t)$ (which is conical at $t\rightarrow 0$) with
any Ricci flow smooth at $t=0$. Let $\sigma(t)$ be a smooth Ricci flow on $\mathcal M$ and $t\in[0,\delta]$. If 
$\sigma(0) <  g_0\ \mathrm{on}\ \mathcal M\setminus\{p_1,\ldots,p_n\}$, then $\sigma(t)\leq \tilde g(t)\ \mathrm{on}\ \mathcal M \ \forall t\in(0,\delta]$.
\end{step}
This step is essentially an application of the maximum principle. Let $s$ be the conformal factor of $\sigma(0)$. Since it is bounded, there exists an $M$ and (by the previous step) a neighbourhood $\Omega$ of the cone points such that $s\leq M\leq u$ for a small time $t<t_1$ on $\Omega$. But since $\tilde g(t)\rightarrow g_0$ and $\sigma(0)<g_0$, for an even smaller time $t<t_2$ we have $\sigma \leq \tilde g$ on the whole $\mathcal M$. Having stablished the inequality for a positive time the maximum principle gives it for any time $t\in(0,\delta)$.

\begin{step}
Comparing two smoothening Ricci flows $\tilde g_1(t)$, $\tilde g_2(t)$ on $t\in(0,\delta]$ with initial metric $g_0$ and
curvature uniformly bounded below, a parabolic rescaling of one of them makes it a smooth Ricci flow even at $t=0$, so it is smaller or
equal than the other. By symmetry, also the other is smaller or equal than the one, so they are identical. 
\end{step}
The point is picking a small $t_0>0$ and the bound $K[\tilde g_1(t)]\geq -\Lambda$. We define a rescaling of $\tilde g_1(t)$ as
$$\sigma(t):=e^{-2\Lambda t_0}\ \tilde g_1(e^{2\Lambda t_0}\ t +t_0)$$
for $t\in [0,(\delta-t_0)e^{-2\Lambda t_0})$. This is a smooth Ricci flow even at $t=0$, and by the lower curvature bound it satisfies $\sigma(0) < g_0$, so by the previous step
$\sigma(t)\leq \tilde g_2(t)$. But moving $t_0\rightarrow 0$ one gets
$\tilde g_1(t) \leq \tilde g_2(t)$
and by symmetry, also 
$\tilde g_2(t) \leq \tilde g_1(t)$.

\end{proof}


\providecommand{\bysame}{\leavevmode\hbox to3em{\hrulefill}\thinspace}
\providecommand{\MR}{\relax\ifhmode\unskip\space\fi MR }
\providecommand{\MRhref}[2]{%
  \href{http://www.ams.org/mathscinet-getitem?mr=#1}{#2}
}
\providecommand{\href}[2]{#2}

\end{document}